\documentclass[11pt]{article}
\usepackage{amssymb}
\usepackage{amsfonts}
\usepackage{amsmath}

\newtheorem{theorem}{Theorem}[section]

\newtheorem{corollary}[theorem]{Corollary}

\newtheorem{definition}[theorem]{Definition}

\newtheorem{lemma}[theorem]{Lemma}

\newtheorem{proposition}[theorem]{Proposition}
\newtheorem{remark}[theorem]{Remark}

\newenvironment{proof}[1][Proof]{\noindent\textbf{#1.} }{\ \rule{0.5em}{0.5em}}
\numberwithin{equation}{section}
\input{tcilatex}

\begin{document}

\title{Matrices for the Weil Representation}
\author{Harold N. Ward \\
Department of Mathematics\\
University of Virginia\\
Charlottesville, VA 22904}
\maketitle

\section{Introduction}

These notes contain a revisit to the matrix construction of the Weil
representation presented in \cite{SP} and \cite{QRC}. Further consequences
of that framework are presented.

\section{Preliminaries and the symplectic algebra\label{SectPrelim}}

As in \cite{SP}, let $q$ be a power of the odd prime $p$, and let $V$ be a
vector space of even dimension $2n$ over the finite field $GF(q)$, endowed
with a nondegenerate symplectic form $\varphi $. The corresponding
symplectic group will be denoted $\mathrm{Sp}(V)$. Let $K$ be $\mathbb{Q}%
(e^{2\pi i/p})$, the field of $p$-th roots of unity. Denote the trace
function $GF(q)\rightarrow GF(p)$ by $\mathrm{tr}$: $\mathrm{tr}(\alpha
)=\alpha +\alpha ^{p}+\ldots +\alpha ^{q/p}$. Then let $\psi $ be the 
\textbf{canonical additive character} of $GF(q)$ given by $\psi (\alpha
)=e^{2\pi i/p\times \mathrm{tr}(\alpha )}$ \cite[p. 190]{LN}. Finally, let $%
f:V\times V\rightarrow K$ be the function $f(x,y)=\psi (\varphi (x,y))$.

\begin{lemma}
\label{Lemfproperties}\emph{\cite[Lemma 1.1]{SP} }For any $x,y,z$ in $V$ we
have

1) $f(x+y,z)=f(x,z)f(y,z).$

2) $f(x,y)^{m}=f(mx,y)=f(x,my)$. Notice that $f(\alpha x,y)=f(x,\alpha y)$
for $\alpha \in GF(q).$

3) $f(y,x)=f(x,y)^{-1}.$

4) $f$ is nondegenerate: if $f(x,y)=1$ for all $y\in V$, then $x=0$.
\end{lemma}

\begin{definition}
\emph{The symplectic algebra of }$V$\emph{\ is the the twisted group ring }$%
A $\emph{\ of the additive group of }$V$\emph{\ over }$K$\emph{, with the
factor set }$f$\emph{. Thus }$A$\emph{\ is a }$K$\emph{-vector space with
basis indexed by the members of }$V$\emph{. The basis element corresponding
to }$x\in V$\emph{\ is denoted }$(x)$\emph{, and the multiplication rule is
that }$(x)(y)=f(x,y)(x+y)$\emph{.}
\end{definition}

\noindent The following properties of $\mathcal{A}$ are straight-forward;
some are in \cite[Proposition 1.2]{SP}.

\begin{lemma}
\label{LemAlgebraFacts}For the algebra $\mathcal{A}$ we have:

1) The identity is $(0).$

2) If $x\in V$, then $(x)$ is a unit and $(x)^{-1}=(-x)$.

3) If $x,y\in V$, then $(x)^{-1}(y)(x)=f(2y,x)(y)$.

4) The center of $\mathcal{A}$ is $K(0)$.

5) For $g\in \mathrm{Sp}(V)$, the map $\sum_{v\in V}\alpha
_{v}(v)\rightarrow \sum_{v\in V}\alpha _{v}(v^{g})$ is a $K$-automorphism of 
$\mathcal{A}$. It will written as $\mathfrak{a}\rightarrow \mathfrak{a}^{g}$.
\end{lemma}

\begin{proof}
Only item 4 needs proof: certainly $K(0)$ is in the center of $\mathcal{A}$,
by item 1. If $x\in V$, then 
\begin{equation*}
(x)^{-1}\left( \sum_{v\in V}\alpha _{y}(v)\right) (x)=\sum_{v\in
V}f(2v,x)\alpha _{y}(v).
\end{equation*}%
Thus $\sum_{v\in V}\alpha _{y}(v)$ is in the center exactly when $%
f(2v,x)\alpha _{v}=\alpha _{v}$ for all $x$ and $v$. But if $v\neq 0$, we
can choose $x$ to make $f(2v,x)\neq 1$, so that $\alpha _{v}$ must be 0.
\end{proof}

These are some general notations we shall use: as in \cite{QRC}, $\chi $ is
the quadratic character on $GF(q)$, and $\delta =\chi (-1)=(-1)^{(q-1)/2}$.
If $U$ is a subspace of $V$, $\mathcal{A}(U)$ is the $K$-subalgebra
generated by the $(u)$, $u\in U$. If $Y$ is a subspace of $X$, then $%
Y^{F}=\left\{ x|F(y,x)=0\text{ for all }y\in Y\right\} $ and $^{F}Y=\left\{
x|F(x,y)=0\text{ for all }y\in Y\right\} $. (However, we shall write $%
Y^{\perp }$ when $F=\varphi $.) The form $F$ is \textbf{nondegenerate} if $%
X^{F}=\left\{ 0\right\} $, and that is equivalent to $\det F\neq 0$. If $F$
is nondegenerate, we also have $^{F}X=\left\{ 0\right\} $. For $A$ an
additive group, $A^{\emptyset }$ is the set of nonzero members of $A$. If
the size of a square matrix needs indicating, it is put on as a subscript.
Finally, if $\mathfrak{m}$ is a linear transformation on a vector space $X$,
then $X^{\mathfrak{m}}=\func{Im}\mathfrak{m}$ and $X_{\mathfrak{m}}=\ker (%
\mathfrak{m}-1)$, the fixed-point subspace of $\mathfrak{m}$.

\begin{remark}
\label{RemFields}\emph{Here are some facts that will be needed in the sequel:%
}
\end{remark}

\begin{enumerate}
\item The trace form $(\alpha ,\beta )\rightarrow \mathrm{tr}(\alpha \beta )$
on $GF(q)$ is nondegenerate \cite[Theorem 5.7]{LN}. It follows from this
that if $X$ is a finite-dimensional vector space over $GF(q)$, then the
functions $x\rightarrow \psi (\mathrm{tr}(l(x)))$, where $l\in X^{\mathrm{lin%
}}$, the space of $GF(q)$-linear functionals on $X$, give the characters of
the additive group of $X$. The standard character summations are%
\begin{eqnarray*}
\sum_{x\in X}\psi (\mathrm{tr}(l(x))) &=&0\text{ if }l\neq 0\text{, and }%
\left\vert X\right\vert \text{ for }l=0; \\
\sum_{l\in X^{l\mathrm{in}}}\psi (\mathrm{tr}(l(x))) &=&0\text{ if }x\neq 0%
\text{, and }\left\vert X\right\vert \text{ for }x=0.
\end{eqnarray*}

\item Let 
\begin{equation*}
\rho =\sum_{\alpha \in GF(q)}\psi (\alpha ^{2})=\sum_{\xi \neq 0}\chi (\xi
)\psi (\xi ),
\end{equation*}%
as in \cite[Section 2]{QRC}. Then $\rho ^{2}=\delta q$. Moreover, for $%
\gamma \neq 0$, $\sum_{\alpha \in GF(q)}\psi (\gamma \alpha ^{2})=\chi
(\gamma )\rho $. If $Q$ is a quadratic form on $X$, then%
\begin{equation*}
\sum_{x\in X}\psi (Q(x))=\chi (Q)\rho ^{\mathrm{rank}Q}q^{\dim X-\mathrm{rank%
}Q}=\chi (Q)\left\vert X\right\vert \delta ^{\dim X-\mathrm{rank}Q}\rho ^{-%
\mathrm{rank}Q}.
\end{equation*}%
Here $\chi (Q)$ means $\chi (\det (Q^{\prime }))$, where $Q^{\prime }$ is
the quadratic form induced by $Q$ on $X/\mathrm{rad}Q$, and we take $\chi
(0)=1$. This classical result can be proved as follows: square the sum
defining $\rho $ to get $\rho ^{2}=\sum_{\alpha ,\beta \in GF(q)}\psi
(\alpha ^{2}+\beta ^{2})$. By \cite[Theorem 6.26]{LN}, the number of
solutions to $\zeta =\alpha ^{2}+\beta ^{2}$ is $q-\delta $ if $\zeta \neq 0$
and $q-\delta +q\delta $ if $\zeta =0$. Thus $\rho ^{2}=(q-\delta
)\sum_{\zeta \in GF(q)}\psi (\zeta )+q\delta \psi (0)$. As $\psi $ is a
nontrivial character, the sum is 0, so that $\rho ^{2}=q\delta \psi
(0)=q\delta $. For the second statement, $\sum_{\alpha \in GF(q)}\psi
(\gamma \alpha ^{2})=\sum_{\alpha \in GF(q)}\psi (\alpha ^{2})$ if $\gamma $
is a square. If not, $\sum_{\alpha \in GF(q)}\psi (\alpha ^{2})+\sum_{\alpha
\in GF(q)}\psi (\gamma \alpha ^{2})=2\sum_{\zeta \in GF(q)}\psi (\zeta )=0$,
making $\sum_{\alpha \in GF(q)}\psi (\gamma \alpha ^{2})=-\sum_{\alpha \in
GF(q)}\psi (\alpha ^{2})=\chi (\gamma )\rho $. The quadratic form result
comes from taking an orthogonal basis for $Q$ on $X$ and using the previous
sums.

\item Continuing, let $F(x)=Q(x)+l(x)$, where $l\in X^{\mathrm{lin}}$. We
wish to evaluate $\sum_{x}\psi (F(x))$. Write $X=X_{0}+\mathrm{rad}Q$ for
some subspace $X_{0}$ complementary to $\mathrm{rad}Q$, and suppose first
that $l$ is not identically $0$ on $\mathrm{rad}Q$. Putting $x=x_{0}+z$,
with $x_{0}\in X_{0}$ and $z\in \mathrm{rad}Q$, we have%
\begin{equation*}
\sum_{x}\psi (F(x))=\sum_{x}\psi (Q(x)+l(x))=\sum_{x_{0}}\psi
(Q(x_{0})+l(x_{0}))\sum_{z}\psi (l(z)).
\end{equation*}%
By item 1, the second sum is $0$, so that $\sum_{x}\psi (F(x))=0$. On the
other hand, if $l$ is $0$ on $\mathrm{rad}Q$, then%
\begin{equation*}
\sum_{x}\psi (Q(x)+l(x))=\left\vert \mathrm{rad}Q\right\vert
\sum_{x_{0}}\psi (Q(x_{0})+l(x_{0})).
\end{equation*}%
If $B$ is the polarization of $Q$, we can find $y_{0}\in X_{0}$ with $%
l(x)=2B(x,y_{0})$ for all $x$. Then $%
Q(x_{0})+l(x_{0})=Q(x_{0}+y_{0})-Q(y_{0})$. Thus%
\begin{eqnarray*}
\sum_{x_{0}}\psi (Q(x_{0})+l(x_{0})) &=&\psi (-Q(y_{0}))\sum_{x_{0}}\psi
(Q(x_{0}+y_{0})) \\
&=&\psi (-Q(y_{0}))\chi (Q)\left\vert X_{0}\right\vert \delta ^{\dim X_{0}-%
\mathrm{rank}Q}\rho ^{-\mathrm{rank}Q} \\
&=&\psi (-Q(y_{0}))\chi (Q)\left\vert X_{0}\right\vert \rho ^{-\mathrm{rank}%
Q}
\end{eqnarray*}%
since $\dim X_{0}=\mathrm{rank}Q$. Thus in this case%
\begin{eqnarray*}
\sum_{x}\psi (F(x)) &=&\left\vert \mathrm{rad}Q\right\vert \sum_{x_{0}}\psi
(Q(x_{0})+l(x_{0})) \\
&=&\psi (-Q(y_{0}))\chi (Q)\left\vert X\right\vert \rho ^{-\mathrm{rank}Q}.
\end{eqnarray*}
\end{enumerate}

These remarks imply the following:

\begin{lemma}
\label{LemQuadraticPsiSum}If $X$ is a finite-dimensional vector space over $%
GF(q)$, and $F(x)=Q(x)+l(x)$, where $Q$ is a quadratic form on $X$ and $l\in
X^{\mathrm{lin}}$ (that is, $F$ is a polynomial of degree at most 2 on $X$,
with $F(0)=0$), then $\left\vert \sum_{x}\psi (F(x))\right\vert \leq
\left\vert X\right\vert $. Equality holds if and only if $F$ is identically
0.
\end{lemma}

\begin{corollary}
\label{CorPsiOne}If $\psi (F(x))=1$ for all $x$, then $F$ is identically 0.
\end{corollary}

\section{The symplectic group in $\mathcal{A}$\label{SectSpinA}}

In this section we produce the embedding of $\mathrm{Sp}(V)$ that is the
main ingredient in the treatment of the Weil representation in \cite{SP}.
However, several other useful results are proved in its development. First
we recall the \textbf{adjoint map} $\mathrm{ad}$ for $\varphi $: it is an
involutory antiautomorphism $\mathfrak{m}\rightarrow \mathfrak{m}^{\mathrm{ad%
}}$ of $\mathrm{End}_{GF(q)}(V)$ with the property that $\varphi (x^{%
\mathfrak{m}},y)=\varphi (x,y^{\mathfrak{m}^{\mathrm{ad}}})$ for all $x,y\in
V$ \cite[Section 36.3]{Go}. The members of $\mathrm{Sp}(V)$ are the
invertible $\mathfrak{m}$ for which $\mathfrak{m}^{\mathrm{ad}}=\mathfrak{m}%
^{-1}$.

\begin{lemma}
\label{LemPerp}If $\mathfrak{m}\in \mathrm{End}_{GF(q)}(V)$, then $\ker 
\mathfrak{m}=\func{Im}(\mathfrak{m}^{\mathrm{ad}})^{\perp }$.
\end{lemma}

\begin{proof}
The equality follows from this sequence of equivalent statements: $x\in 
\func{Im}(\mathfrak{m}^{\mathrm{ad}})^{\perp }$; $\varphi (x,y^{\mathfrak{m}%
^{\mathrm{ad}}})=0$ for all $y$; $\varphi (x^{\mathfrak{m}},y)=0$ for all $y$%
; $x^{\mathfrak{m}}=0$, from the nondegeneracy of $\varphi $; $x\in \ker 
\mathfrak{m}$.
\end{proof}

\noindent Applying this to $\mathfrak{m}=g-1$, $g\in \mathrm{Sp}(V)$, we have

\begin{corollary}
\label{CorPerp}If $g\in \mathrm{Sp}(V)$, then $V_{g}=(V^{g-1})^{\perp }$.
\end{corollary}

\begin{definition}
\label{DefTheta}\emph{Let }$g\in Sp(V)$\emph{. Following Thomas \cite{Th}
and Wall \cite{W}, we introduce the \textbf{theta form} }$\Theta _{g}$\emph{%
\ on }$V^{g-1}$\emph{, defined by}%
\begin{equation*}
\Theta _{g}(x^{g-1},y^{g-1})=\varphi (x^{g-1},y).
\end{equation*}
\end{definition}

\begin{lemma}
\label{LemTheta}\emph{\cite[Lemma 1.1.1 and equation (1.1.3)]{W}} The form $%
\Theta _{g}$ is well-defined and nondegenerate. Its skew-symmetric part is $%
-\varphi /2$ (on $V^{g-1}$), and its symmetric part $B_{g}$ is given by $%
B_{g}(x^{g-1},y^{g-1})=(\varphi (x^{g},y)+\varphi (y^{g},x))/2$. The
corresponding quadratic form $Q_{g}$ has $%
Q_{g}(x^{g-1})=B_{g}(x^{g-1},x^{g-1})=\varphi (x^{g},x)$.
\end{lemma}

\begin{proof}
If $y_{1}^{g-1}=y_{2}^{g-1}$, then $y_{1}-y_{2}\in V_{g}$, and $\varphi
(x^{g-1},y_{1}-y_{2})=0$, by Corollary \ref{CorPerp}. Thus $\varphi
(x^{g-1},y_{1})=\varphi (x^{g-1},y_{2})$, as needed to make $\Theta _{g}$
well-defined. Then if $\Theta _{g}(x^{g-1},y^{g-1})=\varphi (x^{g-1},y)=0$
for all $x$, $y$ must be in $(V^{g-1})^{\perp }$ and hence be a fixed vector
of $g$, whereupon $y^{g-1}=0$. This shows that $\det \Theta _{g}\neq 0$
(relative to any basis of $V^{g-1}$).

The skew-symmetric part of $\Theta _{g}$ has the values%
\begin{eqnarray*}
\frac{1}{2}(\varphi (x^{g-1},y)-\varphi (y^{g-1},x)) &=&\frac{1}{2}(\varphi
(x^{g},y)-\varphi (x,y)-\varphi (y^{g},x)+\varphi (y,x) \\
&=&-\varphi (x,y)+\frac{1}{2}(\varphi (x^{g},y)+\varphi (x,y^{g})).
\end{eqnarray*}%
But 
\begin{eqnarray*}
-\frac{1}{2}\varphi (x^{g-1},y^{g-1}) &=&-\frac{1}{2}(\varphi
(x^{g},y^{g})-\varphi (x,y^{g})-\varphi (x^{g},y)+\varphi (x,y)) \\
&=&-\frac{1}{2}(2\varphi (x,y)-\varphi (x^{g},y)-\varphi (x,y^{g})) \\
&=&-\varphi (x,y)+\frac{1}{2}(\varphi (x^{g},y)+\varphi (x,y^{g})),
\end{eqnarray*}%
the same thing.

For the symmetric part $B_{g}$ of $\Theta _{g}$, we have%
\begin{eqnarray*}
B_{g}(x^{g-1},y^{g-1}) &=&\frac{1}{2}(\varphi (x^{g-1},y)+\varphi
(y^{g-1},x)) \\
&=&\frac{1}{2}(\varphi (x^{g},y)-\varphi (x,y)+\varphi (y^{g},x)-\varphi
(y,x)) \\
&=&\frac{1}{2}(\varphi (x^{g},y)+\varphi (y^{g},x)).
\end{eqnarray*}
\end{proof}

\begin{theorem}
\label{ThmCharSubspace}\emph{\cite[Theorem 1.1.2]{W}} Let $U$ be a subspace
of $V$ endowed with a nondegenerate bilinear form $T$ whose skew-symmetric
part is $-\varphi /2$ on $U$. Then if $B=T+\varphi /2$ and $Q(u)=B(u,u)$, $%
u\in U$, there is an element $g\in \mathrm{Sp}(V)$ for which $U=V^{g-1}$ and 
$\Theta _{g}=T$. If we start with $h\in \mathrm{Sp}(V)$ and take $U=V^{h-1}$
and $T=\Theta _{h}$, then $g=h$. Consequently, the original $g$ is unique.
\end{theorem}

\begin{proof}
Let $s=\sum \psi (Q(u))(u)$ and $\widetilde{s}=\sum \psi (-Q(u))(u)$, both
sums over $u\in U$. For $v\in V$,%
\begin{equation*}
\widetilde{s}(v)s=\sum_{u_{1},u_{2}\in U}\psi (-Q(u_{1}))\psi
(Q(u_{2}))f(u_{1},v)f(u_{1}+v,u_{2})(u_{1}+u_{2}+v).
\end{equation*}%
Put $u_{2}=z-u_{1}$ and use $T(u_{1},z)=B(u_{1},z)-\varphi (u_{1},z)/2$ to
get%
\begin{equation*}
\widetilde{s}(v)s=\sum_{z\in U}\left\{ \psi (Q(z)+\varphi
(v,z))\sum_{u_{1}\in U}\psi (2(\varphi (u_{1},v)-T(u_{1},z)))\right\} (v+z).
\end{equation*}%
Now $u\rightarrow \varphi (u,v)-T(u,z)$ is a linear functional on $U$. So if
it is not identically 0, the inner sum (a character sum) \emph{is} 0. Since $%
T$ is nondegenerate, there is a unique $z=z(v)$ making $\varphi
(u,v)-T(u,z(v))=0$:%
\begin{equation}
T(u,z(v))=\varphi (u,v)\text{ for all }u\in U,v\in V.  \label{eqTphi}
\end{equation}%
The map $v\rightarrow z(v)$ is linear. Thus%
\begin{equation*}
\widetilde{s}(v)s=\left\vert U\right\vert \psi (Q(z(v)+\varphi
(v,z(v)))(v+z(v)).
\end{equation*}%
Moreover, $T(z(v),z(v))=\varphi (z(v),v)$. That is, $Q(z(v))+\varphi
(v,z(v))=0$, and $Q(z(v))=\varphi (z(v),v)$. So%
\begin{equation*}
\widetilde{s}(v)s=\left\vert U\right\vert (v+z(v)).
\end{equation*}%
From $v=0$ we get $\widetilde{s}s=\left\vert U\right\vert (0)$, so that $s$
is invertible and $s^{-1}=\left\vert U\right\vert ^{-1}\widetilde{s}$.
Consequently $g:v\rightarrow v+z(v)$ is nonsingular.

Since $\mathfrak{a}\rightarrow s^{-1}\mathfrak{a}s$ is an automorphism of $%
\mathcal{A}$, $f(v_{1}^{g},v_{2}^{g})=f(v_{1},v_{2})$ for all $v_{1},v_{2}$.
In Lemma \ref{LemQuadraticPsiSum}, let $X=V\oplus V$ and $%
F(v_{1},v_{2})=\varphi (v_{1}^{g},v_{2}^{g})-\varphi (v_{1},v_{2})$. Then $%
\psi (F(v_{1},v_{2})=1$. By Corollary \ref{CorPsiOne}, $F(v_{1},v_{2})=0$,
so that $\varphi (v_{1}^{g},v_{2}^{g})=\varphi (v_{1},v_{2})$, wherewith $%
g\in \mathrm{Sp}(V)$. As $z(v)=v^{g-1}$,%
\begin{equation*}
Q(v^{g-1})=Q(z(v))=\varphi (z(v),v)=\varphi (v^{g-1},v)=\varphi
(v^{g},v)=Q_{g}(v).
\end{equation*}%
Thus $Q=Q_{g}$, and so $B=B_{g}$ and $T=\Theta _{g}$. Certainly $%
V^{g-1}\subseteq U$, since $z(v)\in V^{g-1}$. But given $z\in U$, there is a 
$v$ with $T(z,u)=\varphi (u,v)$ for all $u\in U$, since $\varphi $ is
nonsingular. Then $z=z(v)$, and $U\subseteq V^{g-1}$. So $U=V^{g-1}$.

Finally, let $U=V^{h-1}$ and $T=\Theta _{h}$, for some $h\in \mathrm{Sp}(V)$%
. Then on the one hand, $\Theta _{h}(u,z(v))=\varphi (u,v)$ for all $u\in U$
and all $v\in V$, by (\ref{eqTphi}). On the other hand, $\Theta
_{h}(u,v^{h-1})=\varphi (u,v)$. Since $\Theta _{h}$ is nondegenerate, it
must be that $z(v)=v^{h-1}$, so that $v^{h}=v+z(v)=v^{g}$. Thus $g=h$.
\end{proof}

\begin{corollary}
\label{Cors()Def}\emph{\cite[Proposition 2.1]{SP} }For $g\in \mathrm{Sp}(V)$
there is a unique invertible element $s(g)\in \mathcal{A}$ with $(0)$%
-coefficient 1, for which $s(g)^{-1}(x)s(g)=(x^{g})$ for all $x\in V$. One
has 
\begin{equation}
s(g)=\sum_{y\in V^{g-1}}\psi (Q_{g}(y))(y).  \label{s(g)}
\end{equation}%
If $g,h\in \mathrm{Sp}(V)$, then 
\begin{equation}
s(g)s(h)=\mu (g,h)s(gh)  \label{mu}
\end{equation}%
for some nonzero $\mu (g,h)$ of $K$. (This factor is the inverse of the
\textquotedblleft $\mu (g,h)$\textquotedblright\ in \cite{SP}.) The map $%
g\rightarrow s(g)$ is injective, and%
\begin{equation}
s(g^{-1})=\sum_{y\in V^{g-1}}\psi (-Q_{g}(y))(y).  \label{sinv}
\end{equation}
\end{corollary}

\begin{proof}
The existence, invertibility, and formula for $s(g)$ follow from Theorem \ref%
{ThmCharSubspace}. If $s_{1}$ and $s_{2}$ are two members of $\mathcal{A}$
for which $s_{1}^{-1}(x)s_{1}=s_{2}^{-1}(x)s_{2}$ for all $x$, then $%
s_{1}s_{2}^{-1}$ commutes with all $(x)$, so that $s_{1}=\zeta s_{2}$ for
some nonzero $\zeta \in K$, by Lemma \ref{LemAlgebraFacts}. If the $(0)$%
-coefficients of $s_{1}$ and $s_{2}$ are both 1, then $\zeta =1$ and $%
s_{1}=s_{2}$, giving the uniqueness.

For $g,h\in \mathrm{Sp}(V)$,%
\begin{equation*}
s(h)^{-1}s(g)^{-1}(x)s(g)s(h)=(x^{gh})=s(gh)^{-1}(x)s(gh).
\end{equation*}%
Thus again $s(g)s(h)$ must be a scalar multiple of $s(gh)$, and that
multiple is defined to be $\mu (g,h)$. As the $(0)$-coefficient of $s(gh)$
is 1, we can find $\mu (g,h)$ by computing the $(0)$-coefficient in $%
s(g)s(h) $. Since $f(u,v)=f(-u,-v)$, we may write $s(h)=%
\sum_{y=z^{h}-z}f(z^{h},z)(-y) $. Then the $(0)$-coefficient in $s(g)s(h)$
is $\sum f(x^{g},x)f(z^{h},z)$, the sum over $y\in V^{g-1}\cap V^{h-1}$ with 
$y=x^{g}-x=z^{h}-z$. This coefficient is $\sum_{y\in V^{g-1}\cap
V^{h-1}}\psi (Q_{g}(y)+Q_{h}(y))$. Here $Q_{g}(y)+Q_{h}(y)$ is a quadratic
form on $V^{g-1}\cap V^{h-1}$, by Lemma \ref{LemTheta}, so that item 2 of
Remark \ref{RemFields} implies that $\mu (g,h)$ is a power of $\rho $, up to
sign.

That $g\rightarrow s(g)$ is injective just follows from the fact that $%
s(g)^{-1}(x)s(g)=(x^{g})$, which determines $g$. Since $s(g^{-1})s(g)$ and $%
\widetilde{s}s(g)$ are scalar multiples of $(0)$, where $\widetilde{s}=\sum_{%
\unit{y}\in V^{g-1}}\psi (-Q_{g}(u)(u)$ from the proof of Theorem \ref%
{ThmCharSubspace}, both $s(g^{-1})$ and $\widetilde{s}$ are scalar multiples
of $s(g)^{-1}$ and so of each other. As both have $(0)$-coefficient 1, it
must be that $s(g^{-1})=\widetilde{s}$. Note that 
\begin{equation*}
Q_{g^{-1}}(v)=\varphi (v^{g^{-1}},v)=\varphi (v,v^{g})=-\varphi
(v^{g},v)=-Q_{g}(v).
\end{equation*}
\end{proof}

\begin{proposition}
\label{PropInv}Let $g\in \mathrm{Sp}(V)$\ be an involution, with eigenspaces 
$E_{\varepsilon }=\left\{ v|v^{g}=\varepsilon v\right\} $, $\varepsilon =\pm
1$. Then $Q_{g}=0$ and $s(g)=\sum_{x\in E_{-1}}(x)$. If $g\in \mathrm{Sp}(V)$
and $Q_{g}=0$, then $g$ is an involution.
\end{proposition}

\begin{proof}
We have $V=E_{1}\oplus E_{-1}$, an orthogonal direct sum. For $%
v=v_{1}+v_{-1} $, with $v_{\varepsilon }\in E_{\varepsilon }$, $%
v^{g-1}=2v_{-1}=-v_{-1}^{g-1}$. Thus $V^{g-1}=E_{-1}$, and $%
Q_{g}(v^{g-1})=\varphi (2v_{-1},v)=0$. Then $s(g)=\sum_{x\in E_{-1}}(x)$.

If $Q_{g}=0$ for some $g\in \mathrm{Sp}(V)$, then $\Theta _{g}=-\varphi /2$
on $V^{g-1}$, by Lemma \ref{LemTheta}, so that $\varphi $ is nonsingular on $%
V^{g-1}$, since $\Theta _{g}$ is. We have both $s(g)=\sum_{x\in V^{g-1}}(x)$
and $s(g^{-1})=\sum_{x\in V^{g-1}}(x)$, from Corollary \ref{Cors()Def}. Thus 
$g=g^{-1}$ and $g$ is an involution.
\end{proof}

\begin{proposition}
\label{PropTrans}Let $h_{\gamma }$\ be the transvection $v\rightarrow
v-\gamma ^{-1}\varphi (v,c)c$. Then $s(h_{\gamma })=\sum_{\zeta \in
GF(q)}\psi (\gamma \zeta ^{2})(\zeta c)$\emph{.}
\end{proposition}

\begin{proof}
Here $v^{h_{\gamma }-1}=-\gamma ^{-1}\varphi (v,c)c$, so that $V^{h_{\gamma
}-1}=GF(q)c$. Furthermore,%
\begin{equation*}
Q_{h_{\gamma }}(v^{h-1})=\varphi (v^{h_{\gamma }-1},v)=\varphi (-\gamma
^{-1}\varphi (v,c)c,v)=\gamma ^{-1}\varphi (v,c)^{2}.
\end{equation*}%
\ Thus with $v^{h_{\gamma }-1}=\zeta c$, $Q_{h_{\gamma }}(\zeta c)=\gamma
\zeta ^{2}$, and $s(h_{\gamma })=\sum_{\zeta }\psi (\gamma \zeta ^{2})(\zeta
c)$\emph{.}
\end{proof}

\begin{corollary}
\label{CorSpCentGen}\emph{\cite[Proposition 2.2]{SP}} The centralizer $%
\mathcal{C}$ of $j$ is the $K$-span of the $s(g)$, $g\in \mathrm{Sp}(V)$.
\end{corollary}

\begin{proof}
As in \cite{SP}\emph{,}%
\begin{equation*}
(c)+(-c)=(0)+q^{-1}\sum_{\gamma \neq 0}(\psi (-\gamma )-1)s(h_{\gamma }).
\end{equation*}
\end{proof}

\begin{proposition}
\label{PropDirectSum}Let $h,k\in \mathrm{Sp}(V)$. If $V^{h-1}\cap
V^{k-1}=\left\{ 0\right\} $, that is, the sum $V^{h-1}+V^{k-1}$ is direct,
then $s(h)s(k)=s(hk)$ and $\mu (h,k)=1$. Conversely, if $\mu (h,k)=1$, then $%
V^{h-1}\cap V^{k-1}=\left\{ 0\right\} $. When this happens, $%
V^{hk-1}=V^{h-1}\oplus V^{k-1}$ and $\dim V^{hk-1}=\dim V^{h-1}+\dim V^{k-1}$%
. Moreover, $\Theta _{hk}(V^{h-1},V^{k-1})=0$, and with a basis for $%
V^{hk-1} $ adapted to the direct sum, the matrix $\left[ \Theta _{hk}\right] 
$ for $\Theta _{hk}$ is $\left[ 
\begin{array}{cc}
\left[ \Theta _{h}\right] & 0 \\ 
\ast & \left[ \Theta _{k}\right]%
\end{array}%
\right] $.
\end{proposition}

\begin{proof}
We have $v^{hk-1}=v^{h-1}+v^{h(k-1)}$, showing that $V^{hk-1}\subseteq $ $%
V^{h-1}+V^{k-1}$. Suppose that $V^{h-1}\cap V^{k-1}=\left\{ 0\right\} $.
Then in writing out the product $s(h)s(k)$, there is no collection of terms
with the same basis element $(v)$. Thus $s(hk)=s(h)s(k)$ and $\mu (h,k)=1$.
Each term in $s(h)s(k)$ appears as a product in only one way. So $%
V^{hk-1}=V^{h-1}+V^{k-1}$ and the sum must be direct. If $x=u^{hk-1}\in
V^{hk-1}$ and $x\in V^{h-1}$ too, then in the direct sum $u^{h(k-1)}=0$ and $%
x=u^{h-1}$. Likewise, if $y=v^{hk-1}$ and $y\in V^{k-1}$, then $v^{h-1}=0$,
so that $v\in V_{h}$. Thus%
\begin{equation*}
\Theta _{hk}(x,y)=\varphi (u^{hk-1},v)=\varphi (u^{h-1},v)=0,
\end{equation*}%
from the $\varphi $-orthogonality of $V^{h-1}$ and $V_{h}$.

With $x_{i}=u_{i}^{hk-1}$, where $u_{i}^{h(k-1)}=0$ and $x_{i}=u_{i}^{h-1}$,
we have 
\begin{equation*}
\Theta _{hk}(x_{1},x_{2})=\varphi (u_{1}^{hk-1},u_{2})=\varphi
(u_{1}^{h-1},u_{2})=\Theta _{h}(x_{1},x_{2}),
\end{equation*}%
making $\Theta _{hk}|V^{h-1}=\Theta _{h}$. Similarly, with $%
y_{i}=v_{i}^{hk-1}\in V^{k-1}$, where $v_{i}^{h-1}=0$ and $y_{i}=v_{i}^{k-1}$%
, 
\begin{equation*}
\Theta _{hk}(y_{1},y_{2})=\varphi (v_{1}^{hk-1},v_{2})=\varphi
(v_{i}^{k-1},v_{2})=\Theta _{k}(y_{1},y_{2}).
\end{equation*}%
Hence $\Theta _{hk}|V^{k-1}=\Theta _{k}$. All this implies that by taking a
basis for $V^{hk-1}$ that is the union of one in $V^{h-1}$ followed by one
in $V^{k-1}$, a matrix for $\Theta _{hk}$ can be arranged as displayed.

From the proof of Corollary \ref{Cors()Def}, the $(0)$-coefficient in $%
s(h)s(k)$ is 
\begin{equation*}
\sum_{x\in V^{h-1}\cap V^{k-1}}\psi (Q_{h}(x)+Q_{k}(x)).
\end{equation*}%
As observed there, $Q_{h}+Q_{k}$ is a quadratic form on $V^{h-1}\cap V^{k-1}$%
. If its rank is $r$, then $r\leq \dim (V^{h-1}\cap V^{k-1})$, and the sum,
up to sign, is $\rho ^{r}q^{\dim (V^{h-1}\cap V^{k-1})-r}=\pm \rho ^{2\dim
(V^{h-1}\cap V^{k-1})-r}$, by Remark \ref{RemFields}. That must be $\mu
(h,k) $, and the only way it can be 1 is that $\dim (V^{h-1}\cap V^{k-1})=0$.
\end{proof}

\begin{corollary}
\label{CorFactor}Suppose that $g\in \mathrm{Sp}(V)$ and that $%
V^{g-1}=X\oplus Y$, both terms nonzero. Suppose also that $\Theta
_{g}(X,Y)=0 $, so that $\Theta _{g}|X$ and $\Theta _{g}|Y$ are
nondegenerate. Let $h,k\in \mathrm{Sp}(V)$ with $X=V^{h-1},Y=V^{k-1}$, and $%
\Theta _{g}|X=\Theta _{h},\Theta _{g}|Y=\Theta _{k}$, following Theorem \ref%
{ThmCharSubspace}. Then $g=hk$.
\end{corollary}

\begin{proof}
We have $s(h)=\sum_{x\in X}\psi (Q_{h}(x))(x)$ and $s(k)=\sum_{y\in Y}\psi
(Q_{k}(y))(y)$. Thus%
\begin{eqnarray*}
s(h)s(k) &=&\sum_{x\in X,y\in Y}\psi (Q_{h}(x))\psi (Q_{k}(y))\psi
(f(x,y))(x+y) \\
&=&\sum_{x\in X,y\in Y}\psi (Q_{h}(x)+Q_{k}(y)+\varphi (x,y))(x+y).
\end{eqnarray*}%
The argument of $\psi $ here is 
\begin{equation*}
Q_{g}(x)+Q_{g}(y)+\varphi (x,y)=Q_{g}(x+y)-2B_{g}(x,y)+\varphi (x,y),
\end{equation*}%
and since $\Theta _{g}=B_{g}-\varphi /2$, it is $Q_{g}(x+y)-2\Theta
_{g}(x,y)=Q_{g}(x+y)$. Thus $s(h)s(k)=\sum_{x\in X,y\in Y}\psi
(Q_{g}(x+y))=s(g)$. Since $s(h)s(k)=s(hk)$, from $V^{h-1}\cap
V^{k-1}=\left\{ 0\right\} $, $g=hk$.
\end{proof}

\begin{remark}
\label{RemDirSum}\emph{By the arguments for }$s(g)$\emph{\ we have been
using, if }$g=g_{1}\cdots g_{r}$\emph{\ and }$\dim
V^{g-1}=\sum_{i=1}^{r}\dim V^{g_{i}-1}$\emph{, then }$s(g)=s(g_{1})\cdots
s(g_{r})$\emph{\ and }$V^{g-1}=V^{g_{1}-1}\oplus \cdots \oplus V^{g_{r}-1}$%
\emph{. Moreover, }$\left[ \Theta _{g}\right] $ \emph{is lower
block-triangular, with }$\left[ \Theta _{g_{1}}\right] ,\ldots ,\left[
\Theta _{g_{r}}\right] $ \emph{on the diagonal.}
\end{remark}

\section{$\mathcal{A}$ as a matrix ring\label{SectAMatrixRing}}

For any commutative ring $S$ we let $\mathcal{M}_{m}(S)$ be the $m\times m$
matrix ring over $S$. The goal of this section is to show that the
symplectic algebra $\mathcal{A}$ is isomorphic to $\mathcal{M}_{q^{n}}(K)$.
We prove this by creating a set of matrix units that will be of use later
on. The flow of the argument is quite standard.

Let $W$ and $W^{\ast }$ be two maximal isotropic (\textbf{Lagrangian})
subspaces of $V$ that are complementary, so that $V=W\oplus W^{\ast }$, an
internal direct sum. Then $\varphi $ is identically $0$ on both $W$ and $%
W^{\ast }$, and $\varphi $ sets up a nondegenerate pairing between $W$ and $%
W^{\ast }$. The subalgebras $\mathcal{A}(W)$ and $\mathcal{A}(W^{\ast })$
are simply the group algebras for the two subspaces as additive groups.
Moreover, the characters of the additive group of $W$ are produced by the
maps $w\rightarrow f(w,a)$, $a\in W^{\ast }$. Set $e_{0}=q^{-n}\sum_{w\in
W}(w)$. If $x\in V$, then

\begin{eqnarray*}
e_{0}(x)e_{0} &=&q^{-2n}\sum_{w,w^{\prime }\in W}(w)(x)(w^{\prime }) \\
&=&q^{-2n}\sum_{w,w^{\prime }\in W}f(w-w^{\prime },x)(x+w+w^{\prime }) \\
&=&q^{-2n}\sum_{y,y^{\prime }\in W}f(y,x)(x+y^{\prime }).
\end{eqnarray*}%
As $w\rightarrow f(w,x)$ is a character of $W$, $\sum_{w\in W}f(w,x)=q^{n}$
if $x\in W$, and $\sum_{w\in W}f(w,x)=0$ otherwise. Thus 
\begin{equation}
e_{0}(x)e_{0}=\left\{ 
\begin{tabular}{l}
$e_{0}$ if $x\in W$ \\ 
$0$ if $x\notin W.$%
\end{tabular}%
\right.  \label{e0(x)e0}
\end{equation}%
\newline

Now let $e_{ab}=(-a)e_{0}(b)$, where $a,b\in W^{\ast }$ (we may separate the
two subscripts by a comma when more complicated expressions appear). Then $%
e_{ab}e_{cd}=(-a)e_{0}(b-c)e_{0}(d)$, and by (\ref{e0(x)e0}), this is 0 if $%
b\neq c$, and $e_{ad}$ if $b=c$. Thus the $e_{ab}$ multiply like matrix
units. As customary, put $e_{a}=e_{aa}$, so that $e_{0}$ retains its
original meaning. These $e_{a}$ are orthogonal idempotents in $\mathcal{A(W)}
$. By item 3 of Lemma \ref{LemAlgebraFacts}, $e_{a}=q^{-n}\sum_{w\in
W}f(w,2a)(w)$. Then%
\begin{equation*}
\sum_{a\in W^{\ast }}e_{a}=q^{-n}\sum_{w\in W}\left( \sum_{a\in W^{\ast
}}f(w,2a\right) (w).
\end{equation*}%
The inner sum is 0 if $w\neq 0$, and $q^{n}$ if $w=0$. So $\sum_{a\in
W^{\ast }}e_{a}=(0)$. Furthermore,%
\begin{eqnarray*}
e_{a}(x)e_{b} &=&(-a)e_{0}(a)(x)(-b)e_{0}(b) \\
&=&f(a,x)f(a+x,-b)(-a)e_{0}(a+x-b)e_{0}(b) \\
&=&f(a+b,x)(-a)e_{0}(a+x-b)e_{0}(b).
\end{eqnarray*}%
Thus%
\begin{equation}
e_{a}(x)e_{b}=\left\{ 
\begin{tabular}{l}
$0$ if $a+x-b\notin W$ \\ 
$f(a+b,x)e_{ab}$ if $a+x-b\in W.$%
\end{tabular}%
\right. .  \label{ea(x)eb}
\end{equation}%
Hence $e_{a}(x)e_{b}\in Ke_{ab}$ for all $x$, so that $e_{a}\mathcal{A}%
e_{b}=Ke_{ab}$. Then for $\mathfrak{a}\in \mathcal{A}$,%
\begin{equation*}
\mathfrak{a}=\left( \sum_{a\in W^{\ast }}e_{a}\right) \mathfrak{a}\left(
\sum_{b\in W^{\ast }}e_{b}\right) =\sum_{a,b\in W^{\ast }}e_{a}\mathfrak{a}%
e_{b}=\sum_{a,b\in W^{\ast }}\alpha _{ab}e_{ab},
\end{equation*}%
with the $\alpha _{ab}\in K$. As the number of the $e_{ab}$ is $q^{2n}$,
which is $\dim _{K}\mathcal{A}$, the $e_{ab}$ form a $K$-basis of $\mathcal{A%
}$. Thus we have

\begin{theorem}
\label{ThmMatrix}The algebra $\mathcal{A}$ is indeed isomorphic to $\mathcal{%
M}_{q^{n}}(K)$.
\end{theorem}

\noindent We identify $\mathcal{A}$ with $\mathcal{M}_{q^{n}}(K)$ by this
isomorphism, that is, by means of the matrix units $e_{ab}$. If we wish to
emphasize that we are regarding $\mathfrak{a\in }\mathcal{A}$ as a matrix,
we shall write $[\mathfrak{a}]$.

\begin{remark}
\label{RemTrace}\emph{The trace of the map }$\mathfrak{b}\rightarrow 
\mathfrak{ab}$ \emph{is }$q^{n}\mathrm{tr}\mathfrak{a}$,\emph{\ with }$%
\mathfrak{a,b}\in \mathcal{A}$.\emph{\ If }$x\neq 0$, $\mathfrak{b}%
\rightarrow (x)\mathfrak{b}$\emph{\ has trace }$0$, \emph{so that }$\mathrm{%
tr}[(x)]=0$.\emph{\ Thus}%
\begin{equation}
\mathrm{tr}[(x)]=\left\{ 
\begin{array}{c}
0\text{, }x\neq 0 \\ 
q^{n}\text{, }x=0.%
\end{array}%
\right.  \label{trace}
\end{equation}
\end{remark}

\begin{remark}
\label{RemNewUnits}\emph{Suppose that $\mathcal{M}$ is a matrix algebra over
a field and that $\mathcal{M}\emph{\ }$has two sets of matrix units, the }$%
e_{ab}$ \emph{and the }$e_{ab}^{\prime }$\emph{, indexed by the same finite
set. Then by \cite[Theorem 3, p. 59]{J}, there is an invertible }$u$\emph{\
in $\mathcal{M}$ with }$e_{ab}^{\prime }=u^{-1}e_{ab}u$\emph{, for all }$a,b$%
\emph{. It follows that traces and determinants of members of $\mathcal{M}$
computed from the two sets of units are the same.}
\end{remark}

\section{The Weil representation\label{SectWeilRep}}

Continue now with the set-up and notation of Section \ref{SectAMatrixRing}.
We have $s(-1)=\sum_{v\in V}(v)$, by Proposition \ref{PropInv}. As in \cite%
{SP}, we set $j=q^{-n}s(-1)$, an involution. To determine $[j]$, observe that%
\begin{eqnarray*}
je_{0} &=&q^{-2n}\sum_{v\in V,w\in W}(v)(w) \\
&=&q^{-2n}\sum_{v\in V,w\in W}f(v,w)(v+w) \\
&=&q^{-2n}\sum_{u\in V}\left( \sum_{w\in W}f(u,w)\right) (u).
\end{eqnarray*}%
Once again, the inner sum is 0 if $u\notin W$ and $q^{n}$ if $u\in W$. So%
\begin{equation*}
je_{0}=e_{0}.
\end{equation*}%
Thus 
\begin{eqnarray*}
e_{a}je_{b} &=&(-a)e_{0}(a)j(-b)e_{0}(b) \\
&=&(-a)e_{0}(a)(b)je_{0}(b) \\
&=&(-a)e_{0}(a+b)e_{0}(b).
\end{eqnarray*}%
This last is 0 unless $b=-a$, when it is just $e_{a,-a}$.

Partition $W^{\ast }$ as $P\cup -P\cup \{0\}$ in some way, and on ordering $%
P $, order $-P$ so that $-a<-b$ in $-P$ when $a<b$ in $P$. Then label rows
and columns with the members of $W^{\ast }$ following that arrangement. This
makes%
\begin{equation*}
\lbrack j]=\left[ 
\begin{array}{ccc}
0 & I & 0 \\ 
I & 0 & 0 \\ 
0 & 0 & 1%
\end{array}%
\right] ,
\end{equation*}%
where $I$ is the identity matrix of size $(q^{n}-1)/2$ and the various 0s
have appropriately matching sizes. Direct computations give the following
result:

\begin{proposition}
\label{PropjCent}Let $\mathcal{C}$ be the centralizer of $j$ in $\mathcal{A}$%
. Then the members of $\mathcal{C}$ are the matrices of the form%
\begin{equation}
\left[ 
\begin{array}{ccc}
A & B & \mathbf{b} \\ 
B & A & \mathbf{b} \\ 
\mathbf{a} & \mathbf{a} & \alpha%
\end{array}%
\right] ,  \label{centmatrix}
\end{equation}%
the blocks matching those of $\left[ j\right] $. The map given by%
\begin{equation*}
\left[ 
\begin{array}{ccc}
A & B & \mathbf{b} \\ 
B & A & \mathbf{b} \\ 
\mathbf{a} & \mathbf{a} & \alpha%
\end{array}%
\right] \rightarrow \left[ 
\begin{array}{ccc}
A-B & 0 & 0 \\ 
0 & A+B & 2\mathbf{b} \\ 
0 & \mathbf{a} & \alpha%
\end{array}%
\right]
\end{equation*}%
provides an isomorphism of $\mathcal{C}$ with the direct sum $\mathcal{M}%
_{(q^{n}-1)/2}(K)\oplus \mathcal{M}_{(q^{n}+1)/2}(K)$.
\end{proposition}

The representation $g\rightarrow s(g)$ is a projective representation of $%
\mathrm{Sp}(V)$, with factor set $\mu $. A key result in any development of
the Weil representation is that it is equivalent to an ordinary
representation. Thus we seek a function $\eta :\mathrm{Sp}(V)\rightarrow K$
such that with $t(g)=\eta (g)s(g)$, the representation $g\rightarrow t(g)$
is ordinary; that is, $\mu (g,h)=\eta (gh)\eta (g)^{-1}\eta (h)^{-1}$, $\mu $
as in (\ref{mu}). Now 
\begin{equation*}
\mu (g,h)^{q^{n}}\det s(gh)=\det s(g)\det s(h).
\end{equation*}%
Let $\flat $ be the homomorphism%
\begin{equation*}
\left[ 
\begin{array}{ccc}
A & B & \mathbf{b} \\ 
B & A & \mathbf{b} \\ 
\mathbf{a} & \mathbf{a} & \alpha%
\end{array}%
\right] \rightarrow A-B
\end{equation*}%
of $\mathcal{C}$ onto $\mathcal{M}_{(q^{n}-1)/2}(K)$, as inferred from
Proposition \ref{PropjCent}. Then%
\begin{equation*}
\mu (g,h)^{(q^{n}-1)/2}\det s(gh)^{\flat }=\det s(g)^{\flat }\det
s(h)^{\flat }.
\end{equation*}%
Consequently:

\begin{lemma}
\label{LemEta}If we define%
\begin{equation}
\eta (g)=(\det s(g)^{\flat })^{2}(\det s(g))^{-1},  \label{eqEta}
\end{equation}%
we have the needed equation $\mu (g,h)=\eta (gh)\eta (g)^{-1}\eta (h)^{-1}$.
\end{lemma}

\begin{remark}
\label{RemEtaUnique}\emph{The comments in Remark \ref{RemNewUnits} show that 
}$\eta $\emph{\ is well-defined: it does not depend on the choices made in
obtaining the matrix realization of }$A$\emph{. In fact, if we also have }$%
\mu (g,h)=\eta ^{\prime }(gh)\eta ^{\prime }(g)^{-1}\eta ^{\prime }(h)^{-1}$%
\emph{, then }$\eta ^{\prime }\eta ^{-1}$\emph{\ is a homomorphism of }$%
Sp(V) $\emph{\ to }$C^{\varnothing }$\emph{. As }$Sp(V)^{\prime }=Sp(V)$%
\emph{\ except when }$n=1$\emph{\ and }$q$\emph{\ (odd) is 3, }$\eta
^{\prime }=\eta $\emph{\ with that one exception.}
\end{remark}

For example, $j^{\flat }=-I_{(q^{n}-1)/2}$, so that $\det (j^{\flat
})=(-1)^{(q^{n}-1)/2}=\det (j)$. Then 
\begin{equation*}
\det (j^{\flat })^{2}\det (j)^{-1}=\det (j)=(-1)^{(q^{n}-1)/2}.
\end{equation*}%
Since $(q^{n}-1)/2\equiv n(q-1)/2\,(\bmod\ 2)$, $%
(-1)^{(q^{n}-1)/2}=(-1)^{n(q-1)/2}=\delta ^{n}$. As $s(-1)=q^{n}j$, we get 
\begin{equation*}
\eta (-1)=\delta ^{n}(q^{n})^{2((q^{n}-1)/2)-q^{n}}=\delta ^{n}q^{-n}.
\end{equation*}%
Consequently, 
\begin{equation}
t(-1)=\delta ^{n}j.  \label{t(-1)}
\end{equation}

\begin{definition}
\label{DefWeilRep}\emph{The representation }$\mathcal{W}$\emph{\ of }$Sp(V)$%
\emph{\ given by }$g\rightarrow t(g)$\emph{, with degree }$q^{n}$\emph{, is
the \textbf{Weil representation}. By Proposition \ref{PropjCent}, }$\mathcal{%
W}$\emph{\ has the two constituents }$\mathcal{W}_{-}:g\rightarrow
t(g)^{\flat }$\emph{\ and }$\mathcal{W}_{+}:g\rightarrow t(g)^{\sharp }$%
\emph{, of degrees }$(q^{n}-1)/2$\emph{\ and }$(q^{n}+1)/2$\emph{,
respectively. Since }$C$\emph{\ is spanned by the }$t(g)$\emph{, from
Corollary \ref{CorSpCentGen}, both of these representations are irreducible.
We denote the characters of these three representations by }$\omega $\emph{, 
}$\omega _{-}$\emph{, and }$\omega _{+}$\emph{,\ so that }$\omega =\omega
_{-}+\omega _{+}$\emph{.}
\end{definition}

\section{The Weil character\label{SectWeil1}}

Before examining the Weil character in detail, we record this implication of
the matrix picture above:

\begin{proposition}
\label{PropOmegaMinus}For the Weil characters,%
\begin{equation}
\omega _{-}(g)=\frac{\omega (g)-\delta ^{n}\omega (-g)}{2}.
\label{omegaminus}
\end{equation}%
Thus%
\begin{equation}
\omega _{+}(g)=\frac{\omega (g)+\delta ^{n}\omega (-g)}{2}.
\label{omegaplus}
\end{equation}
\end{proposition}

\begin{proof}
If $t(g)=\left[ 
\begin{array}{ccc}
A & B & \mathbf{b} \\ 
B & A & \mathbf{b} \\ 
\mathbf{a} & \mathbf{a} & \alpha%
\end{array}%
\right] $, then $t(g)^{\flat }=A-B$ and%
\begin{equation*}
t(-g)=t(-1)t(g)=\delta ^{n}\left[ 
\begin{array}{ccc}
0 & I & 0 \\ 
I & 0 & 0 \\ 
0 & 0 & 1%
\end{array}%
\right] \left[ 
\begin{array}{ccc}
A & B & \mathbf{b} \\ 
B & A & \mathbf{b} \\ 
\mathbf{a} & \mathbf{a} & \alpha%
\end{array}%
\right] =\delta ^{n}\left[ 
\begin{array}{ccc}
B & A & \mathbf{b} \\ 
A & B & \mathbf{b} \\ 
\mathbf{a} & \mathbf{a} & \alpha%
\end{array}%
\right] .
\end{equation*}%
Thus%
\begin{eqnarray*}
\frac{\omega (g)-\delta ^{n}\omega (-g)}{2} &=&\frac{(2\mathrm{tr}A+\alpha
)-\delta ^{n}(\delta ^{n}(2\mathrm{tr}B+\alpha ))}{2} \\
&=&\mathrm{tr}A-\mathrm{tr}B \\
&=&\omega _{-}(g).
\end{eqnarray*}
\end{proof}

To begin with, from $t(g)=\eta (g)s(g)$ and $\mathrm{tr}s(g)=q^{n}$, we have
for the Weil character $\omega $,%
\begin{equation}
\omega (g)=q^{n}\eta (g).  \label{omegabyeta}
\end{equation}%
For example,%
\begin{equation}
\omega (-1)=\delta ^{n}.  \label{omega(-1)}
\end{equation}

To obtain formulas for $\omega $, we start with any $g\in \mathrm{Sp}(V)$
for which $W^{g}=W$. Let $g^{\ast }$ be the map induced by $g$ on $W^{\ast }$%
: if $a\in W^{\ast }$, $W+a^{g^{\ast }}=W+a^{g}$. For $[s(g)]$, we have $%
e_{0}^{g}=e_{0}$, so $e_{0}s(g)e_{a}=s(g)e_{0}e_{a}=0$ unless $a=0$;
similarly $e_{a}s(g)e_{0}=0$ if $a\neq 0$. Thus%
\begin{equation*}
\lbrack s(g)]=\left[ 
\begin{array}{ccc}
A & B & 0 \\ 
B & A & 0 \\ 
0 & 0 & \alpha%
\end{array}%
\right] .
\end{equation*}%
Here%
\begin{equation*}
\alpha e_{0}=e_{0}s(g)e_{0}=\sum_{y=x^{g-1}\in
V^{g-1}}f(x^{g},x)e_{0}(y)e_{0}=\sum_{y=x^{g-1}\in W\cap
V^{g-1}}f(x^{g},x)e_{0},
\end{equation*}%
the last from (\ref{e0(x)e0}). Thus%
\begin{equation}
\alpha =\sum_{y=x^{g-1}\in W\cap V^{g-1}}f(x^{g},x).  \label{gamma}
\end{equation}%
We have $e_{a}^{g}=(-a^{g})e_{0}(a^{g})$. But $a^{g}=w+a^{g^{\ast }}$ for
some $w\in W$, so 
\begin{eqnarray*}
(a^{g}) &=&(w+a^{g^{\ast }})=f(a^{g^{\ast }},w)(w)(a^{g^{\ast }}), \\
(-a^{g}) &=&f(w,a^{g^{\ast }})(-a^{g^{\ast }})(-w).
\end{eqnarray*}%
Therefore%
\begin{eqnarray*}
e_{a}^{g} &=&f(a^{g^{\ast }},w)f(w,a^{g^{\ast }})(-a^{g^{\ast
}})(-w)e_{0}(w)(a^{g^{\ast }}) \\
&=&e_{a^{g^{\ast }}}.
\end{eqnarray*}%
Thus with $\left[ s(g)\right] =\left[ \alpha _{ab}\right] $,%
\begin{equation*}
\alpha _{ab}e_{ab}=e_{a}s(g)e_{b}=s(g)e_{a^{g^{\ast }}}e_{b}.
\end{equation*}%
This can be nonzero only when $b=a^{g^{\ast }}$. Therefore $[s(g)]$ is
monomial, and it follows that both $A+B$ and $A-B$ are monomial. Moreover, $%
\det (A+B)=(-1)^{\beta }\det (A-B)$, where $\beta $ is the number of nonzero
entries in $B$. This is the number of members $a$ of $P$ for which $%
a^{g^{\ast }}\in -P$. The sign $(-1)^{\beta }$ is evaluated in the next
lemma.

\begin{lemma}
\label{LemmaSign}Let $Z$ be a vector space of dimension $n>0$ over $GF(q)$, $%
q$ odd. Partition $Z$ into $P\cup -P\cup \{0\}$, as above. Let $\sigma
(g)=(-1)^{\left\vert P^{g}\cap (-P)\right\vert }$ for $g$ a nonsingular
linear transformation on $Z$. Then%
\begin{equation*}
\sigma (g)=\chi (\det g).
\end{equation*}
\end{lemma}

\begin{proof}
Index $q^{n}-1$ independent variables $X_{z}$ by the nonzero members of $Z$,
and in the $\mathbb{Z}$-ring of polynomials $R$ in the $X_{z}$, let $\pi
=\dprod\limits_{z\in P}(X_{z}-X_{-z})$. Let $g$ act on $R$ by $%
X_{z}\longrightarrow X_{z^{g}}$. Then $\pi ^{g}=\sigma (g)\pi $, and it
follows that $g\longrightarrow \sigma (g)$ is a homomorphism from $GL(Z)$ to 
$\left\langle -1\right\rangle $. If $\sigma $ is nontrivial, then it must be
that $\sigma (g)=\chi (\det g)$, since $GL(Z)/GL(Z)^{\prime }$ is cyclic of
order $q-1$ \cite[Theorem 4.7]{A}.

For the nontriviality of $\sigma $, observe first that if $P^{\prime }\cup
-P^{\prime }\cup \{0\}$ is a second partition of $Z$ and $\pi ^{\prime }$
the corresponding product, then $\pi ^{\prime }=\varepsilon \pi $ for some
sign, $\varepsilon $. It follows that $\pi ^{\prime g}=\sigma (g)\pi
^{\prime }$, so that $\sigma (g)$ does not depend on the partition used. Now
realize $Z$ as $GF(q^{n})$ and let $\alpha $ be a primitive element. Let $%
P=\left\{ 1,\alpha ,\ldots ,\alpha ^{(q^{n}-3)/2}\right\} $; $\alpha
^{(q^{n}-1)/2}=-1$. Then $g:z\longrightarrow \alpha z$ has $\left\vert
P^{g}\cap (-P)\right\vert =1$, $\sigma (g)=-1$, and $\sigma $ is indeed
nontrivial. By way of corroboration, this $g$ has determinant $\alpha
+\alpha ^{q}+\ldots +\alpha ^{q^{n-1}}=\alpha ^{(q^{n}-1)/(q-1)}$ (the norm
of $\alpha $), a primitive element of $GF(q)$. Thus $\chi (\alpha
^{(q^{n}-1)/(q-1)})=-1=\sigma (g)$.
\end{proof}

Now the matrix $\left[ 
\begin{array}{ccc}
A & B & 0 \\ 
B & A & 0 \\ 
0 & 0 & \alpha%
\end{array}%
\right] $ is equivalent to $\left[ 
\begin{array}{ccc}
A-B & B-A & 0 \\ 
0 & A+B & 0 \\ 
0 & 0 & \alpha%
\end{array}%
\right] $, with no scalings involved. Thus 
\begin{equation*}
\det \left[ s(g)\right] =(-1)^{\beta }\alpha \det (A-B)^{2}=\chi (\det
g^{\ast })\alpha \det (A-B)^{2},
\end{equation*}%
by the lemma. From $\eta (g)=(\det s(g)^{\flat })^{2}(\det s(g))^{-1}$
(equation (\ref{eqEta})) and $\det s(g)^{\flat }=\det (A-B)$, we conclude:

\begin{theorem}
\label{ThmEta}If $g\in \mathrm{Sp}(V)$ and $W^{g}=W$, then%
\begin{equation*}
\eta (g)=\frac{\chi (\det g^{\ast })}{\sum\limits_{y=x^{g-1}\in W\cap
V^{g-1}}f(x^{g},x)}.
\end{equation*}
\end{theorem}

\begin{corollary}
\label{CorInvolution}Suppose that $g\in \mathrm{Sp}(V)$ is an involution.
Let $E_{\varepsilon }$ be the eigenspace of $g$ for eigenvalue $\varepsilon $
and let $\dim E_{-1}=2m$. Then $\eta (g)=\delta ^{m}q^{-m}$ and $\omega
(g)=\delta ^{m}q^{n-m}$.
\end{corollary}

\begin{proof}
As in Proposition \ref{PropInv}, $\varphi $ is nonsingular on $E_{-1}$, so
that $\dim E_{-1}$ is indeed even. From Remark \ref{RemTrace} we can tailor
the choice of $W$ and $W^{\ast }$ to $g$. Thus we set up $W$ and $W^{\ast }$
by taking decompositions $E_{\varepsilon }=W_{\varepsilon }\oplus
W_{\varepsilon }^{\ast }$, with $W=W_{1}\oplus W_{-1}$, $W^{\ast
}=W_{1}^{\ast }\oplus W_{-1}^{\ast }$. Here $\dim W_{-1}=\dim W_{-1}^{\ast
}=m$, and $\dim W_{1}=\dim W_{1}^{\ast }=n-m$. We have $V^{g-1}=E_{-1}$, and 
$Q_{g}=0$. Thus in $\left[ s(g)\right] $, $\alpha =\left\vert W\cap
V^{g-1}\right\vert =\left\vert W_{-1}\right\vert =q^{m}$. On $W^{\ast }$, $%
a\rightarrow a^{g^{\ast }}$ has fixed point space $W_{1}^{\ast }$ and so $%
(q^{n}-q^{n-m})/2=q^{n-m}(q^{m}-1)/2$ two-cycles. Hence the sign of $\pi
_{g} $ is $(-1)^{(q^{m}-1)/2}$. As noted before, this is $\delta ^{m}$. Thus 
$\eta (g)=\delta ^{m}q^{-m}$. Once again, we get $\eta (-1)=\delta
^{n}q^{-n} $.
\end{proof}

\begin{corollary}
\label{CorInvOmegaMinus}For such an involution $g$,%
\begin{equation*}
\omega _{-}(g)=\frac{\delta ^{m}q^{n-m}-\delta ^{n-m}q^{m}}{2}.
\end{equation*}
\end{corollary}

\begin{proof}
Here $-g$ is also an involution with $\dim E_{-1}=2n-2m$, and Proposition %
\ref{PropOmegaMinus} applies.
\end{proof}

\begin{corollary}
\label{CorTransvection}Let $h$ be the transvection $v\rightarrow v-\gamma
^{-1}\varphi (v,c)c$, where $\gamma \neq 0$ and $c\neq 0$. Then $\eta
(h)=\rho ^{-1}\chi (\gamma )$ and $\omega (h)=q^{n}\rho ^{-1}\chi (\gamma )$.
\end{corollary}

\begin{proof}
Now choose $W$ so that $c\in W$. From Proposition \ref{PropTrans},%
\begin{equation*}
\sum_{\substack{ y\in W\cap V^{h-1}  \\ y=x^{h-1}}}f(x^{h},x)=\sum_{\beta
\in GF(q)}\psi (\beta ^{2}\gamma ).
\end{equation*}%
By Remark \ref{RemFields}, this is $\chi (\gamma )\rho $. On $W^{\ast }$, $%
h^{\ast }$ is the identity. Thus $\eta (h)=\rho ^{-1}\chi (\gamma )$.
\end{proof}

With these results in hand, we can now prove a version of the formula in 
\cite[Remark 1.3]{Th}.

\begin{theorem}
\label{ThmEtaFormula}Let $g\in \mathrm{Sp}(V)$ with theta form $\Theta _{g}$%
. Then 
\begin{equation*}
\eta (g)=\rho ^{-\dim V^{g-1}}\chi (\Theta _{g}).
\end{equation*}
\end{theorem}

\begin{proof}
Recall that $\chi (\Theta _{g})$ means $\chi (\det \Theta _{g})$, $\det
\Theta _{g}$ computed relative to any basis. The two examples we have so far
corroborate the formula: for an involution $g$, $V^{g-1}=V_{-1}$, and with $%
\dim V_{-1}=2m$, Corollary \ref{CorInvolution} gives $\eta (g)=\delta
^{m}q^{-m}$. As we saw in Proposition \ref{PropInv}, $Q_{g}=0$ and $\Theta
_{g}=-\varphi /2$, making $\det \Theta _{g}=(-1/2)^{2m}$ and $\chi (\Theta
_{g})=1$. Moreover, $q=\delta \rho ^{2}$ makes $\eta (g)=\delta ^{m}\delta
^{m}\rho ^{-2m}=\rho ^{-r}$, as asserted. Similarly, for $h$ the
transvection $h:v\rightarrow v-\gamma ^{-1}\varphi (v,c)c$, $V^{h-1}=GF(q)c$%
, and $\Theta _{h}(\xi c,\zeta c)=\xi \zeta \gamma $, with $\det \Theta
_{h}=\gamma $, all by Proposition \ref{PropTrans}. From Corollary \ref%
{CorTransvection}, $\eta (h)=\rho ^{-1}\chi (\gamma )$, again in agreement.

We can thus induct on $\dim V^{g-1}$. Suppose that $V^{h-1}\cap
V^{k-1}=\left\{ 0\right\} $ for $h,k\in \mathrm{Sp}(V)$, neither $h$ nor $k$
the identity. Then Proposition \ref{PropDirectSum} implies that $\det \Theta
_{hk}=\det \Theta _{h}\det \Theta _{k}$. Because $\mu (h,k)=1$, it follows
that $\eta (hk)=\eta (h)\eta (k)$, and by induction,%
\begin{equation}
\eta (hk)=\rho ^{-\dim V^{h-1}}\chi (\Theta _{h})\rho ^{-\dim V^{k-1}}\chi
(\Theta _{k}).  \label{EtaProd}
\end{equation}%
But $\dim V^{hk-1}=\dim V^{h-1}+\dim V^{k-1}$, as shown in the proof of
Proposition \ref{PropDirectSum}, and then $\eta (hk)=\rho ^{-\dim
V^{hk-1}}\chi (\Theta _{hk})$ indeed.

To obtain such a factoring, we may assume that a given $g$ is neither an
involution nor a transvection, for both of which the formula is already
correct. Thus the quadratic form $Q_{g}$ on $V^{g-1}$ is not identically 0,
by Proposition \ref{PropInv}. Take $c\in V^{g-1}$ with $Q_{g}(c)\neq 0$. Put 
$X=\left\langle c\right\rangle $ and $Y=\left\langle c\right\rangle ^{\Theta
_{g}}$; $Y\neq \left\{ 0)\right\} $. Corollary \ref{CorFactor} applies,
since $Q_{g}(c)\neq 0$ implies that $X\cap Y=\left\{ 0)\right\} $. Then $%
g=hk $, where $X=V^{h-1}$ and $Y=V^{k-1}$, and neither $h$ nor $k$ is 1. The
product result (\ref{EtaProd}) now gives the asserted formula for $g$.
\end{proof}

\section{Computations\label{SectComputations}}

\subsection{Scaling\label{SubsScaling}}

Let $\nu $ be a nonsquare in $GF(q)$ and let $\varphi ^{\prime }=\nu \varphi 
$. Construct the algebra $\mathcal{A}^{\prime }$ with the procedure used for 
$\mathcal{A}$ but with $\varphi ^{\prime }$ in place of $\varphi $. Label
corresponding ingredients for $\mathcal{A}^{\prime }$ with the dash. The
symplectic group for $\varphi ^{\prime }$ is the same as that for $\varphi $%
. But the form in Definition \ref{DefTheta} changes to $\Theta _{g}^{\prime
}=\nu \Theta _{g}$. Then $\chi (\Theta _{g}^{\prime })=(-1)^{\dim
V^{g-1}}\chi (\Theta _{g})$, and 
\begin{equation*}
\eta ^{\prime }(g)=\rho ^{-\dim v^{g-1}}\chi (\Theta _{g}^{\prime })=(-\rho
)^{\dim V^{g-1}}\chi (\Theta _{g}),
\end{equation*}
from Theorem \ref{ThmEtaFormula}. That is, $\eta ^{\prime }$ comes from $%
\eta $ by changing $\rho $ to $-\rho $. The same holds for $\omega ^{\prime
} $ from $\omega $. As pointed out when this scaling is discussed in \cite[%
Section 3]{QRC}, $-\rho $ may not be an algebraic conjugate of $\rho $.

\subsection{Permuting $V$\label{SubsPermute}}

\begin{proposition}
\label{PropPermRep}(compare \cite[Theorem 4.4(c)]{G}) The homomorphism $%
g\rightarrow t(g^{-1})^{T}\otimes t(g)\in \mathcal{M}_{q^{2n}}(K)$ is the
permutation representation of $\mathrm{Sp}(V)$ on $V$. Consequently $\omega
(g^{-1})\omega (g)=\left\vert V_{g}\right\vert $.
\end{proposition}

\begin{proof}
This is immediate from the fact that $t(g^{-1})(v)t(g)=(v^{g})$ for $v\in V$.
\end{proof}

This leads to a corroboration of Theorem \ref{ThmEtaFormula}: we have 
\begin{equation*}
\omega (g)=\delta ^{n}\rho ^{2n-\dim V^{g-1}}\chi (\Theta _{g})
\end{equation*}%
and 
\begin{equation*}
\omega (g^{-1})=\delta ^{n}\rho ^{2n-\dim V^{g^{-1}-1}}\chi (\Theta
_{g^{-1}}).
\end{equation*}%
Now $v^{g-1}=-v^{g(g^{-1}-1)}$ implies that $V^{g-1}=V^{g^{-1}-1}$. If $%
u,v\in V^{g-1}$ and $v=y^{g-1}=(-y^{g})^{g^{-1}-1}$, then $\Theta
_{g}(u,v)=\varphi (u,y)$ and $\Theta _{g^{-1}}(u,v)=-\varphi (u,y^{g})$.
Thus $\Theta _{g^{-1}}(u,v)=-\Theta _{g}(u,v^{g})$. Hence $\det \Theta
_{g^{-1}}=(-1)^{\dim V^{g-1}}\det (g|V^{g-1})\det \Theta _{g}$. But since $%
v^{g}\equiv v(\bmod\ V^{g-1})$ and $\det g=1$, we have $\det (g|V^{g-1})=1$.
So $\det \Theta _{g^{-1}}=(-1)^{\dim V^{g-1}}\det \Theta _{g}$. Then 
\begin{equation*}
\chi (\Theta _{g^{-1}})=\chi ((-1)^{\dim V^{g-1}})\chi (\Theta _{g})=\delta
^{\dim V^{g-1}}\chi (\Theta _{g}).
\end{equation*}%
Therefore 
\begin{eqnarray*}
\omega (g^{-1})\omega (g) &=&\delta ^{2n}(\rho ^{2})^{2n-\dim V^{g-1}}\delta
^{\dim V^{g-1}}\chi (\Theta _{g})^{2} \\
&=&(\delta q)^{2n-\dim V^{g-1}}\delta ^{\dim V^{g-1}}=q^{2n-\dim V^{g-1}} \\
&=&\left\vert V_{g}\right\vert ,
\end{eqnarray*}%
as $V_{g}=(V^{g-1})^{\perp }$.

\subsection{Embedding\label{SubsEmbed}}

From the \textquotedblleft internal\textquotedblright\ formula for $\eta $\
in Theorem \ref{ThmEtaFormula}, we can prove the following: let $V=U\oplus
U^{\prime }$, with $U^{\prime }=U^{\perp }$; $\varphi $\ is nondegenerate on 
$U$ and $U^{\prime }$. Then $\mathrm{Sp}(U)\times \mathrm{Sp}(U^{\prime })$\
embeds into $\mathrm{Sp}(V)$: if $g\in \mathrm{Sp}(U)$ and $g^{\prime }\in 
\mathrm{Sp}(U^{\prime })$, then $(u+u^{\prime })^{g\times g^{\prime
}}=u^{g}+(u^{\prime })^{g^{\prime }}$, where $u\in U$ and $u^{\prime }\in
U^{\prime }$. We have $V^{g\times g^{\prime }-1}=U^{g-1}\oplus (U^{\prime
})^{g^{\prime }-1}$. and $\Theta _{g\times g^{\prime }}(u_{1}+u_{1}^{\prime
},u_{2}+u_{2}^{\prime })=(\Theta _{U})_{g}(u_{1},u_{2})(\Theta _{U^{\prime
}})_{g^{\prime }}(u_{1}^{\prime },u_{2}^{\prime })$, the $\Theta $s on the
right taken for $U$ and $U^{\prime }$. Then from the formula, $\eta (g\times
g^{\prime })=\eta _{U}(g)\eta _{U^{\prime }}(g^{\prime })$ and $\omega
(g\times g^{\prime })=\omega _{U}(g)\omega _{U^{\prime }}(g^{\prime })$.

\subsection{Eigenvalue restrictions and semisimplicity\label{SubsSemi}}

\begin{proposition}
\label{PropEta(g-1)}If no eigenvalue of $g$ is 1, that is, $g-1$\ is
invertible, then%
\begin{equation}
\omega (g)=\delta ^{n}\chi (\det (g-1)).  \label{g-1 inv omega}
\end{equation}
\end{proposition}

\begin{proof}
For such a $g$, $V^{g-1}=V$. With $y_{2}=x_{2}^{g-1}$, we have $\Theta
_{g}(y_{1},y_{2})=\varphi (y_{1},y_{2}^{(g-1)^{-1}})$. In matrix terms, $%
\Theta _{g}(y_{1},y_{2})=y_{1}\Phi (y_{2}(g-1)^{-1})^{T}=y_{1}\Phi
(g-1)^{-T}y_{2}^{T}$, $\Phi $\ being the matrix for $\varphi $. Thus $\det
\Theta _{g}=\det (\Phi (g-1)^{-T})=\det \Phi \det (g-1)^{-1}$. But relative
to a symplectic basis, $\det \Phi =1$. Hence $\chi (\Theta _{g})=\chi (\det
(g-1))$. So 
\begin{equation}
\eta (g)=\rho ^{-2n}\chi (\det (g-1))=\delta ^{n}q^{-n}\chi (\det (g-1)),
\label{g-1 inv eta}
\end{equation}%
from which the stated formula follows. Once again, $\eta (-1)=\delta
^{n}q^{-n}$, because $\chi (\det (-2I_{2n}))=\chi (-2)^{2n}=1$.
\end{proof}

\begin{corollary}
\label{CorOmegaMinusTrans}Let $h$ be the usual transvection $v\rightarrow
v-\gamma ^{-1}\varphi (v,c)c$, where $\gamma \neq 0$ and $c\neq 0$. Then%
\begin{equation*}
\omega _{-}(h)=\frac{q^{n}\rho ^{-1}\chi (\gamma )-1}{2}.
\end{equation*}
\end{corollary}

\begin{proof}
All eigenvalues of $h$ are 1, so those of $-h$ are $-1$. Thus $\det
(-h-1)=(-1)^{2n}=1$, and $\omega (-h)=\delta ^{n}$. The formula follows from
Propositions \ref{PropOmegaMinus} and \ref{PropEta(g-1)}.
\end{proof}

\begin{corollary}
\label{CorSemi}If $g\in \mathrm{Sp}(V)$ is semisimple, then $\omega
(g)=\omega ^{\prime }(g)$.
\end{corollary}

\begin{proof}
Here $\omega ^{\prime }$ is the Weil character for $\mathcal{A}^{\prime }$
from Subsection \ref{SubsScaling}, and the restriction on $g$ means that $g$
is diagonalizable in some extension of $GF(q)$. If $E$ is the eigenspace of $%
g$ for eigenvalue 1, then the semisimplicity implies that $\varphi $ is
nonsingular on $E$ and $g|E$ is the identity. With the decomposition $%
V=E\oplus E^{\perp }$, the results in Subsection \ref{SubsEmbed} apply, $g$
being identified with $(g|E)\times (g|E^{\perp })$. For $g|E^{\perp }$,
Proposition \ref{PropEta(g-1)} gives $\omega _{E^{\perp }}(g|E^{\perp
})=\delta ^{n-m}\chi (\det (g|E^{\perp }-1))$, where $\dim E=2m$. As $\omega
_{E}(g|E)=q^{m}$, we get $\omega (g)=q^{m}\delta ^{n-m}\chi (\det
(g|E^{\perp }-1))$. There being no $\rho $ here, we also have the same
formula for $\omega ^{\prime }(g)$, and $\omega ^{\prime }(g)=\omega (g)$.
\end{proof}

\subsection{Constructions\label{SubsConstructions}}

One way to set up $V$\ and $\varphi $\ is to take $V=GF(q^{2n})$\ and to let 
$\varphi (x,y)=\mathrm{tr}(\varepsilon xy^{q^{n}})$, where the trace is from 
$GF(q^{2n})$\ to $GF(q)$\ and $\varepsilon ^{q^{n}}=-\varepsilon $. We have 
\begin{equation*}
\varphi (y,x)=\mathrm{tr}(\varepsilon yx^{q^{n}})=\mathrm{tr}(\varepsilon
^{q^{n}}y^{q^{n}}x)=-\mathrm{tr}(\varepsilon xy^{q^{n}})=-\varphi (x,y),
\end{equation*}%
as needed, and the nondegeneracy of $\varphi $\ follows from the
nondegeneracy of the trace form itself. If $z\in V$\ and $z^{q^{n}+1}=1$,
then the multiplication $g:x\rightarrow zx$\ is a member of $Sp(V)$. With $%
z\neq 1$, Proposition \ref{PropEta(g-1)} applies. As in Lemma \ref{LemmaSign}%
, the determinant of $x\rightarrow ax$\ is the norm $a^{(q^{2n}-1)/(q-1)}$\
of $a$. Moreover, $\chi (\alpha )=\alpha ^{(q-1)/2}$\ for $\alpha \in GF(q)$%
. Thus 
\begin{eqnarray*}
\chi (\det (g-1)) &=&\left( (z-1)^{\frac{q^{2n}-1}{q-1}}\right) ^{\frac{q-1}{%
2}}=\left( (z-1)^{q^{n}-1}\right) ^{\frac{q^{n}+1}{2}} \\
&=&\left( \frac{1/z-1}{z-1}\right) ^{\frac{q^{n}+1}{2}}=\left( \frac{-1}{z}%
\right) ^{\frac{q^{n}+1}{2}} \\
&=&-(-1)^{\frac{q^{n}-1}{2}}\left( \frac{1}{z}\right) ^{\frac{q^{n}+1}{2}%
}=-\delta ^{n}z^{\frac{q^{n}+1}{2}}.
\end{eqnarray*}%
(Once again, we use $(-1)^{(q^{n}-1)/2}=\delta ^{n}$.) Now $z\rightarrow
z^{(q^{n}+1)/2}$\ is the \textquotedblleft quadratic
character\textquotedblright\ $\chi _{Z}$\ on the cyclic subgroup $Z$\ of
order $q^{n}+1$\ in $GF(q^{2n})^{\emptyset }$: $\chi _{Z}(z)=1$\ if $z$\ is
a square in $Z$, $-1$\ if not. So $\chi (\det (g-1))=-\delta ^{n}\chi
_{Z}(z) $, and $\omega (g)=\delta ^{n}\chi (\det (g-1))=-\chi _{Z}(z)$.

If $z\neq -1$, we can apply this to $g^{\prime }=-g$ and get 
\begin{eqnarray*}
\chi (\det (g^{\prime }-1)) &=&-\delta ^{n}\chi _{Z}(-z)=-\delta
^{n}(-1)^{(q^{n}+1)/2}\chi _{Z}(z) \\
&=&\delta ^{n}(-1)^{(q^{n}-1)/2}\chi _{Z}(z)=\chi _{q^{n}+1}(z).
\end{eqnarray*}%
So $\omega (-g)=\delta ^{n}\chi _{Z}(z)$. If $g$\ (and $z$) has order $o(g)$%
, then $\chi _{Z}(z)=(-1)^{(q^{n}+1)/o(g)}$. Thus for $z\neq \pm 1$,%
\begin{equation*}
\omega _{-}(g)=-\chi _{Z}(z)=(-1)^{(q^{n}+1)/o(g)}.
\end{equation*}%
by \ref{PropOmegaMinus}.

A second construction for $V$\ and $\varphi $\ takes $V=GF(q^{n})\oplus
GF(q^{n})$, with $\varphi ((x,y),(x^{\prime },y^{\prime }))=\mathrm{tr}%
(xy^{\prime }-yx^{\prime })$, the trace now from $GF(q^{n})$\ to $GF(q)$.
This time the map $g:(x,y)\rightarrow (zx,z^{-1}y)$, for $z\in
GF(q^{n})^{\emptyset }$, is in $Sp(V)$. Taking $z\neq 1$, we have 
\begin{eqnarray*}
\det (g-1) &=&(z-1)^{\frac{q^{n}-1}{q-1}}(z^{-1}-1)^{\frac{q^{n}-1}{q-1}} \\
&=&(-z^{-1})^{\frac{q^{n}-1}{q-1}}((z-1)^{2})^{\frac{q^{n}-1}{q-1}}.
\end{eqnarray*}%
Then 
\begin{eqnarray*}
\chi (\det (g-1)) &=&(-1)^{\frac{q^{n}-1}{2}}(z^{-1})^{\frac{q^{n}-1}{2}%
}(z-1)^{q^{n}-1} \\
&=&\delta ^{n}\chi _{Z}(z),
\end{eqnarray*}%
with $\chi _{Z}$\ now the quadratic character on $GF(q^{n})$. Thus $\omega
(g)=\delta ^{n}\chi (\det (g-1))=\chi _{Z}(z)$. Similarly, if $z\neq -1$ and 
$g^{\prime }=-g$, $\chi (\det (g^{\prime }-1))=\delta ^{n}\chi _{Z}(-z)=\chi
_{Z}(z)$, and $\omega (-g)=\delta ^{n}\chi _{Z}(z)$. So this time, if $z\neq
\pm 1$, $\omega _{-}(g)=\chi _{Z}(z)(1-\delta ^{n}\times \delta ^{n})/2=0$.

\subsection{Some characters of $\mathrm{SL}(2,q)$}

With the various character values in hand, we can write out the part of the
character table for $\mathrm{Sp}(2,q)(\simeq \mathrm{SL}(2,q))$ that
involves $\omega ,\omega _{-}$, and $\omega _{+}$(see the table in \cite[%
Section 38]{Do}, for example, based on work of I. Schur). Let $\left[ 
\begin{array}{cc}
0 & 1 \\ 
-1 & 0%
\end{array}%
\right] $ be the matrix for $\varphi $. The transvection $h:v\rightarrow
v-\gamma ^{-1}\varphi (v,(0,1))(0,1)$ we have been using has matrix $\left[ 
\begin{array}{cc}
1 & -\gamma ^{-1} \\ 
0 & 1%
\end{array}%
\right] $. From Corollary \ref{CorTransvection}, $\omega (h)=\delta \rho
\chi (\gamma )$, and by \ref{CorOmegaMinusTrans}, $\omega (-h)=\delta $.%
\begin{equation*}
\begin{tabular}[t]{c|cccccc}
& $1$ & $-1$ & $\left[ 
\begin{array}{cc}
1 & \alpha \\ 
0 & 1%
\end{array}%
\right] $ & $\left[ 
\begin{array}{cc}
-1 & \beta \\ 
0 & -1%
\end{array}%
\right] $ & $%
\begin{array}{c}
o(g)|(q-1) \\ 
g\neq \pm 1%
\end{array}%
$ & $%
\begin{array}{c}
o(g)|(q+1) \\ 
g\neq \pm 1%
\end{array}%
$ \\ \hline
$\omega _{-}$ & $\frac{q-1}{2}$ & $-\delta \frac{q-1}{2}$ & $\frac{\rho \chi
(\alpha )-1}{2}$ & $\frac{-\rho \chi (\beta )+\delta }{2}$ & $0$ & $%
-(-1)^{(q^{n}+1)/o(g)}$ \\ 
$\omega _{+}$ & $\frac{q+1}{2}$ & $\delta \frac{q+1}{2}$ & $\frac{\rho \chi
(\alpha )+1}{2}$ & $\frac{\rho \chi (\beta )+\delta }{2}$ & $%
(-1)^{(q^{n}-1)/o(g)}$ & $0$ \\ 
$\omega $ & $q$ & $\delta $ & $\rho \chi (\alpha )$ & $\delta $ & $%
(-1)^{(q^{n}-1)/o(g)}$ & $-(-1)^{(q^{n}+1)/o(g)}$%
\end{tabular}%
\end{equation*}

Using $\mathcal{A}^{\prime }$ from Subsection \ref{SubsScaling}, we can add
two more irreducible characters $\omega _{-}^{\prime }$ and $\omega
_{+}^{\prime }$ to this table, with the same rows as for $\omega _{-}$ and $%
\omega _{+}$ but with $-\rho $ in place of $\rho $.

\section{Factorizations\label{SectFactorizations}}

The embedding of $\mathrm{Sp}(V)$ into $\mathcal{A}$ can be used to study
factorizations of group elements. The best known is the fact that every
member of $\mathrm{Sp}(V)$ is a product of transvections. Initially one
shows that the transvections generate $\mathrm{Sp}(V)$ without paying
attention to how many transvections are needed in a product for a given
group member \cite[Theorem 8.5]{T}. That number was examined by Dieudonn\'{e}
\cite{D}, and we shall deal with it in this section.

\begin{theorem}
\label{ThmTransProd}\emph{\cite[Theorem 2]{D}} If $g\in \mathrm{Sp}(V)$ and $%
g$ is not an involution, then $g$ is the product of $\dim V^{g-1}$
transvections. If $g$ is an involution, then $g$ is the product of $\dim
V^{g-1}+1$ transvections. As usual, $1$ is declared to be an empty product.
\end{theorem}

\begin{proof}
If $g=g_{1}\cdots g_{r}$ with transvections $g_{i}$, then $V^{g-1}\subseteq
V^{g_{1}-1}+\cdots +V^{g_{r}-1}$. Since $\dim V^{g_{i}-1}=1$, by Proposition %
\ref{PropTrans}, $\dim V^{g-1}\leq r$; so there must be at least $\dim
V^{g-1}$ transvections in the product. Suppose that $g$ is not an
involution. Then $Q_{g}$ is not identically 0, by Proposition \ref{PropInv}.
When $\dim V^{g-1}=1$, $g$ is already a transvection. Thus we can assume
that $\dim V^{g-1}>1$. If there is $c\in V^{g-1}$ for which $Q_{g}(c)=\gamma
\neq 0$ and for which $Q_{g}$ is not identically 0 on $Y=\left\langle
c\right\rangle ^{\Theta _{g}}$, then with $X=GF(q)c$ in Corollary \ref%
{CorFactor}, $g=hk$, where $h$ is the transvection $v\rightarrow v-\gamma
^{-1}\varphi (v,c)c$ and $Y=V^{k-1}$. Here $k$ is also not an involution
since $Q_{k}=Q_{g}|Y$ is not identically 0. Induction then implies that $k$
is the product of $\dim Y=\dim V^{g-1}-1$ transvections, making $g$ the
product of $\dim V^{g-1}$ transvections.

For the needed $c$, note that the map $\left\langle c\right\rangle
\rightarrow \left\langle c\right\rangle ^{\Theta _{g}}$ is one-to-one
between lines and hyperplanes of $V^{g-1}$, because $\Theta _{g}$ is
nondegenerate. Suppose that there were three different hyperplanes in $%
V^{g-1}$ that are totally singular for $Q_{g}$. Then the intersection of any
two of them must be the radical, $\mathrm{rad}Q_{g}$, of $Q_{g}$. The
induced quadratic form on the two-dimensional space $V^{g-1}/\mathrm{rad}%
Q_{g}$ would then have at least three singular lines. However, nonsingular
planes have either none or two \cite[pp. 138--139]{T}. Thus in general, $%
V^{g-1}$ has at most two $Q_{g}$-totally singular hyperplanes. But since $%
\dim V^{g-1}>1$, there are at least three lines $\left\langle c\right\rangle 
$ with $Q_{g}(c)\neq 0$, as is easy to see. Therefore we can find the
required $c$ with $Q_{g}(c)\neq 0$ and $Q_{g}$ not identically 0 on $%
\left\langle c\right\rangle ^{\Theta _{g}}$.

Now let $g$ be an involution, with $\dim V^{g-1}=2l$. It cannot be that $%
g=g_{1}\cdots g_{2l}$ for transvections $g_{i}$. For if so, it must be that $%
V^{g-1}=V^{g_{1}-1}\oplus \cdots \oplus V^{g_{2l}-1}$, by the dimensions. If
we set $h=g_{1}$ and $k=g_{2}\cdots g_{2l}$ in Remark \ref{RemDirSum}, then $%
\Theta _{g}|V^{g_{1}-1}=\Theta _{g_{1}}$. But the left is 0, since $Q_{g}=0$
and $\Theta _{g}=-\varphi /2$ on $V^{g-1}$. Yet the right is not 0, by
Proposition \ref{PropTrans}. Thus a transvection product for $g$ must have
more than $2l$ factors.

If $c\in V^{g-1}$, $c\neq 0$, then $c^{g}=-c$. Let $h$ be the transvection $%
v\rightarrow $. Then $gh=hg$: for the successive images of $v$ by $gh$ are $%
v\rightarrow v^{g}\rightarrow v^{g}-\alpha \varphi (v^{g},c)c$; and by $hg$
are $v\rightarrow v-\alpha \varphi (v,c)c\rightarrow v^{g}+\alpha \varphi
(v,c)c$. But this last is indeed $v^{g}+\alpha \varphi
(v^{g},c^{g})c=v^{g}-\alpha \varphi (v^{g},c)c$. Thus $(gh^{-1})^{2}=h^{-2}$
and $gh^{-1}$ is not an involution. So it is a product $g_{1}\cdots g_{r}$
of transvections, with$r\leq \dim V^{g-1}$, since $V^{gh-1}\subseteq V^{g-1}$%
. Then $g=hg_{1}\cdots g_{r}$, and now it must be that $r=\dim V^{g-1}=2l$.
Notice the freedom of choice for $h$.
\end{proof}



\end{document}